\documentclass[a4paper,12pt]{amsart}
\usepackage{amsmath,amssymb,amsthm,enumerate,
textcomp,mathabx}
\usepackage[T2A]{fontenc}
\usepackage[utf8]{inputenc}
\usepackage{cite}

\usepackage{graphicx}

\DeclareMathOperator{\rank}{rk}
\DeclareMathOperator{\Ker}{Ker}
\DeclareMathOperator{\Der}{Der}
\DeclareMathOperator{\Frac}{Frac}
\DeclareMathOperator{\ad}{ad}

\newtheorem{theorem}{Theorem}
\newtheorem{lemma}{Lemma}

\newtheorem{corollary}{Corollary}
\newtheorem{note}{Remark}
\newtheorem{example}{Example}

\begin{document}
\sloppy
\title[Locally nilpotent Lie algebras of derivations of integral domains]
{Locally nilpotent Lie algebras of derivations of integral domains}
\author
{A.P.Petravchuk, O.M.Shevchyk, and  K.Ya.Sysak}
\address{A. P. Petravchuk:
Department of Algebra and Mathematical Logic, Faculty of Mechanics and Mathematics, Kyiv
Taras Shevchenko University, 64, Volodymyrska street, 01033  Kyiv, Ukraine}
\email{aptr@univ.kiev.ua , apetrav@gmail.com}
\address{O. M. Shevchyk:
Department of Algebra and Mathematical Logic, Faculty of Mechanics and Mathematics, Kyiv
Taras Shevchenko University, 64, Volodymyrska street, 01033  Kyiv, Ukraine}
\email{oshev4ik@gmail.com}
\address{K. Ya. Sysak:
Department of Algebra and Mathematical Logic, Faculty of Mechanics and Mathematics, Kyiv
Taras Shevchenko University, 64, Volodymyrska street, 01033  Kyiv, Ukraine}
\email{sysakkya@gmail.com}
\date{\today}
\keywords{Lie algebra, locally nilpotent,  derivation, triangular Lie algebra, integral domain}
\subjclass[2010]{primary 17B66; secondary 17B05, 13N15}

\begin{abstract}
Let $\mathbb K$ be a field of characteristic zero and  $A$  an integral domain over $\mathbb K.$
The Lie algebra  $\Der_{\mathbb K} A$ of all $\mathbb K$-derivations of $A$ carries very important information about the algebra $A.$ This Lie algebra is embedded into the Lie algebra $R\Der_{\mathbb K} A\subseteq \Der_{\mathbb K}R$, where $R={\rm Frac}(A)$  is the fraction field of $A.$ The rank $rk_{R}L$ of a subalgebra $L$ of $R\Der_{\mathbb K} A$ is defined as dimension  $\dim_R RL.$
We prove that every locally nilpotent subalgebra $L$ of $R\Der_{\mathbb K} A$  with $rk _{R}L=n$ has a series of ideals $0=L_0\subset L_1\subset L_2\dots \subset L_n=L$ such that $\rank_R L_i=i$ and all the quotient Lie algebras $L_{i+1}/L_{i}, i=0, \ldots , n-1,$ are abelian. We also describe  all maximal (with respect to inclusion) locally nilpotent subalgebras $L$ of the Lie algebra $R\Der_{\mathbb K} A$ with $rk _{R}L=3.$
\end{abstract}

\maketitle

\section{Introduction}

Let $\mathbb K$ be a field of characteristic zero and  $A$  an associative-commutative algebra over $\mathbb K$ that is an integral domain. The set of all $\mathbb K$-derivations of $A$ forms a Lie algebra $\Der_{\mathbb K} A$, which carries important (and often exhaustive) information about the algebra $A$ (see, for example, \cite{Siebert}). In the  case of the formal power series ring $A=\mathbb R[[x_1,x_2,\dots,x_n]],$ the structure of subalgebras of the Lie algebra $\Der_{\mathbb K}A$ is closely connected with the structure of the symmetry groups of differential equations.
Finite-dimensional subalgebras of the Lie algebra $\Der_{\mathbb K} A$, where $A=\mathbb K[[x]]$, $A=\mathbb K[[x,y]]$ and $\mathbb K=\mathbb R$ or $\mathbb K=\mathbb C$,  are described in \cite{Olver},\cite{Olver1}, \cite{Lie}.

Each derivation $D\in \Der_{\mathbb K} A$ can be  uniquely extended  to a derivation of the fraction field $R=\Frac(A)$ of $A$, and
if $r\in R$ then one can define a derivation $rD \colon R \to R$ by setting $rD(x)=r\cdot D(x)$ for all $x\in R$.
For the study of the Lie algebra $\Der_{\mathbb K} A$, it is convenient to consider a larger Lie algebra $R\Der_{\mathbb K} A$. It is an $R$-linear hull of the set $\{ rD | r\in R, D\in \Der_{\mathbb K} A\}$ and simultaneously a subalgebra (over $\mathbb K$) of the Lie algebra $\Der_{\mathbb K} R$ of all derivations of $R$.
We will denote the Lie algebra $R\Der_{\mathbb K} A$ by $W(A)$. For a subalgebra $L$ of $W(A)$ we define the rank $\rank_R L$ of $L$ over $R$ as $\rank_R L=\dim_R RL$. In \cite{MP1}, nilpotent and solvable subalgebras of finite rank of the Lie algebra $W(A)$ were studied. The structure of nilpotent Lie algebras of derivations with rank $3$ was described in \cite{PetravchukADM}. Nilpotent subalgebras of $W(A)$ with the center of large rank were characterized in \cite{SysakADM}.

In this paper, we study  locally nilpotent  subalgebras $L$ of the Lie algebra $ W(A)$ with  $\rank_R L=n$ over the fraction field $R$. In particular, we prove in Theorem 1 that $L$ contains a series of ideals
$$0=L_0\subset L_1\subset L_2\dots \subset L_n=L$$
such that  $\rank_R L_i=i$ and all the quotient Lie algebras $L_{i+1}/L_{i}$ are abelian for $i=0, 1, \dots, n-1$. Theorem 2 describes maximal (with respect to inclusion) locally nilpotent subalgebras of rank $3$ of the Lie algebra $W(A).$  Note that subalgebras of rank $1$  in $W(A)$ are one-dimensional over their field of constants. Lemma~\ref{Uzh1} describes the structure of subalgebras of rank $2$ from $W(A)$ obtained in \cite{Uzh}.  The Lie algebras $u_n(\mathbb K)$ of triangular derivations of polynomial rings, which were investigated in \cite{Bavula}, may be the reference point for the study of locally nilpotent Lie algebras of derivations.

We use standard notation. The ground field $\mathbb K$ is arbitrary of  characteristic zero.  By $R$ we denote the fraction field of the integral domain $A$.
The Lie algebra $$R\Der_{\mathbb K} A=\langle rD \ | \  r\in R, D\in \Der_{\mathbb K} A \rangle$$ is denoted by $W(A)$. A $\mathbb K$-linear hull of elements $x_1, x_2,\dots, x_n$ we write by $\mathbb K \langle x_1,x_2,\dots, x_n \rangle$.
Let $L$ be a subalgebra of $W(A)$. Then the subfield $F=F(L)$ of the field $R$ that consists of all $r\in R$ with  $D(r)=0$ for all $D\in L$ is called the field of constants for $L$. The rank $\rank_R L$ of $L$ over $R$ is defined by $\rank_R L:=\dim_R RL$, where $RL=R\langle rD| r\in R, D\in L\rangle$. If $I$ is an ideal of $L$ such that $I=RI\cap L$, then one can define the rank (over $R$) of the quotient Lie algebra $L/I$ as $\rank_R L/I:=\dim_R RL/RI$. By $u_n(\mathbb K)$ we denote the Lie algebra of all triangular derivations of the polynomial ring $\mathbb K[x_1,x_2,\dots,x_n]$. This algebra consists of all derivations of the form $$D=f_1(x_2,x_3,\dots, x_n)\frac{\partial}{\partial x_1} +f_2(x_3,x_4,\dots, x_n)\frac{\partial}{\partial x_2} + \dots +f_n\frac{\partial}{\partial x_n},$$ where $f_i\in \mathbb K[x_{i+1},\dots, x_n]$, $i=1,2,\dots, n-1$, and $f_n\in \mathbb K$.
 A Lie algebra is called locally nilpotent if every its finitely generated subalgebra is nilpotent. The Lie algebra $u_n(\mathbb K)$ is locally nilpotent but not nilpotent. It contains a series of ideals $0=I_0\subset I_1\subset \dots \subset I_n=u_n(\mathbb K)$ with abelian factors and $\rank_R I_s=s$ for all $s=0,1,\dots, n$ (see \cite{Bavula}).
Let $V$ be a vector space over $\mathbb K$ (not necessary finite dimensional) and $T$ a linear operator on $V$. The operator  $T$ is called locally nilpotent if for each $v\in V$ there exists a number $n=n(v)\geq 1$ such that $T^n(v)=0$.

\section{Series of ideals in locally nilpotent Lie algebras of derivations}
Some auxiliary results are presented in the following lemmas.
\begin{lemma}\cite[Lemma 1]{MP1}\label{product}
Let $D_1,\ D_2\in W(A)$ and $a,\ b\in R$. Then $$[aD_1,bD_2]=ab[D_1,D_2]+aD_1(b)D_2-bD_2(a)D_1.$$
\end{lemma}
As we mentioned above, the set $RL$ is an $R$-linear hull of elements $rD$ for all $r\in R$ and $D\in L$. Analogously we define the set $FL$ for the field of constants $F=F(L)$.
\begin{lemma}\label{MakPetr}
Let $L$ be a subalgebra of the Lie algebra $W(A)$ and $F$ the field of constants for $L$. Then:
\begin{enumerate}
\item[(1)] \cite[Lemma 2]{MP1}
 $FL$ and $RL$ are $\mathbb K$-subalgebras of the Lie algebra $W(A)$. Moreover, $FL$ is a Lie algebra over $F$, and if $L$ is abelian, nilpotent or solvable, then $FL$ has the same property respectively.
\item[(2)]\cite[Lemma 4]{MP1}\label{ideal}
If $I$ is an ideal of the Lie algebra $L$, then the vector space $RI\cap L$ over $\mathbb K$ is also an ideal of $L$.
\item[(3)]\cite[Theorem 1]{MP1}
If $L$ is a nilpotent subalgebra  of  $W(A)$ of finite rank over $R$, then the Lie algebra $FL$ is finite-dimensional over the field of constants $F$.
\item[(4)]\cite[Proposition 1]{MP1}
Let $L$ be a nilpotent subalgebra of $W(A)$. If $\rank_R L=1$, then $L$  is abelian and $\dim_F FL=1$. If $\rank_R L=2$ and $\dim _{\mathbb K}L\geq 3$, then there exist $D_1, \ D_2\in FL$ and $a\in R$ such that $FL=F\langle D_1, aD_1,\dots, a^k/k! D_1,D_2\rangle$, where $[D_1,D_2]=0$ and $D_1(a)=0$, $D_2(a)=1$.
\end{enumerate}
\end{lemma}

\begin{lemma}\cite[Lemmas 5, 8]{Uzh}\label{loc_nilp}
Let $L$ be a locally nilpotent subalgebra of rank $n$ over $R$ of the Lie algebra $W(A)$ and $F$ the field of constants for $L$. Then
\begin{enumerate}
\item[(1)] The Lie algebra $FL$  over $F$  is locally nilpotent and $\rank_R FL=n$.
\item[(2)] If the derived Lie algebra $L'=[L,L]$ is of rank $k$ over $R,$ then $M=RL' \cap L$ is an ideal of $L$ such that $\rank_R M=\rank_R L'$ and $FL/FM$ is an abelian Lie algebra with $\dim _{F}(FL/FM)=n-k.$
\end{enumerate}
\end{lemma}

\begin{lemma}\cite[Lemma 5]{MP1}\label{series_nilp_case}
Let $L$ be a nilpotent subalgebra  of rank $n>0$ over $R$ of the Lie algebra $W(A)$ and $F$ the field of constants for $L$. Then $L$ contains a series of ideals  $0=I_0\subset I_1\subset \dots\subset I_{n-1}\subset I_n=L$ such that $\rank_R I_s=s$ and $[I_s, I_s]\subseteq I_{s-1}$ for all $s=1,\dots, n$. Moreover, $\dim_F FL/FI_{n-1}=1$.
\end{lemma}

\begin{lemma}\label{abelian_factor}
Let $L$ be a locally nilpotent subalgebra of the Lie algebra  $W(A)$. Let $L_1\subseteq L_2$ be subalgebras of $L$ such that $L_1=RL_1\cap L_2$ is an ideal of $L_2$. If $\rank_R L_2/L_1=1$, then $L_2/L_1$ is an abelian quotient Lie algebra.
\end{lemma}
\begin{proof}
Let $D+L_1$ be a nonzero element of $L_2/L_1$. Then each element of $L_2/L_1$ is of the form $rD+L_1$ for some $r\in R$. The elements $D$ and $rD$ generate a nilpotent  subalgebra  $L_3=\mathbb K\langle D, rD\rangle$ of the Lie  algebra $L$ since $L$ is  locally nilpotent. Every nilpotent subalgebra of  rank 1 over $R$ from $W(A)$ is abelian (Lemma~\ref{MakPetr} (4)). Thus $L_3$ is an abelian Lie algebra.
Then $$[D+L_1, rD+L_1]=[D,rD]+L_1\subseteq L_1.$$
Since $D$, $rD$ are arbitrarily chosen, $[L_2/L_1, L_2/L_1]\subseteq L_1$ and the quotient Lie algebra $L_2/L_1$ is abelian.
\end{proof}

\begin{note}\label{note}
Let $L$ be a subalgebra of finite rank over $R$ of the Lie algebra $W(A)$ and  $I$  a proper ideal of $L$ such that $I=RI\cap L$. Then $\rank_R L >\rank_R I.$
\end{note}

\begin{lemma}\label{factor_rank}
Let $L$ be a locally nilpotent subalgebra of finite rank over $R$ of  the Lie algebra $W(A)$. Let $I$ be an ideal of  $L$ such that $I=RI\cap L.$  If the quotient Lie algebra $L/I$ is nonzero, then $\rank_R (L/I)'<\rank_R(L/I).$
\end{lemma}
\begin{proof}
Suppose to the contrary that there exist a subalgebra $L$ of $W(A)$ and an ideal $I$ of $L$ that satisfy the conditions of the lemma, and
$\rank_R(L/I)'=\rank_R (L/I).$ Then this rank equals  $n-k,$ where  $\rank_R L=n$, $\rank_R I=k.$  Choose a basis $\{\overline D_1, \overline D_2,  \ldots, \overline D_{n-k}\}$ of $(L/I)'$ as the set of vectors  over $R.$ Under our assumptions this basis is also a basis of $L/I$ over $R$. Each $\overline D_i\in  (L/I)'$ is a sum of some commutators from $L/I$, that is
$$\overline D_i=\sum_{j=1}^{k_i} [\overline S_j^{(i)},\overline T_j^{(i)}]=\sum_{j=1}^{k_i} [S_j^{(i)},T_j^{(i)}]+I, \ \ \ k_i\geq 1, \ i=1,2,\dots, n-k$$ for some $\overline S_j^{(i)}, \ \overline T_j^{(i)}\in L/I$. Let us denote by $N$ the subalgebra of $L$ generated  by representatives $ S_j^{(i)}, \ T_j^{(i)}$ of cosets  $\overline S_j^{(i)}, \ \overline T_j^{(i)}$, $j=1,\dots,k_i$, $i=1,\dots,n-k$. Since the Lie algebra $L$ is locally nilpotent, the subalgebra $N$ is nilpotent. Denote $L_1=N+I$. It is easy to see that
$$ \rank_R (L_1/I)=\rank_R (L_1/I)'=n-k=\rank_R(L/I).$$ This implies the equalities $\rank_R L_1=\rank_R L_1'=n$.
Since $L_1/I=N+I/I\simeq N/(N\cap I)$ is a nilpotent Lie  algebra, the center $Z(L_1/I)$ is nonzero. Let us choose a nonzero $D_1+I\in Z(L_1/I)$ and denote  $J_1=R(D_1+I)\cap L_1$. Then $J_1$ is an ideal of $L_1$ by Lemma~\ref{MakPetr} (2). Since $RI\cap L=I$ and $D_1\not \in I$, we get $\rank_R J_1=k+1$. If $k+1<n$, then we take nonzero $D_2+J_1\in Z(L_1/J_1)$ and consider $J_2=R(D_2+J_1)\cap L_1$. The ideal $J_2$ of $L_1$ is of rank $k+2$ over $R$ and $RJ_2\cap L_1=J_2$. By a continuation of these arguments, we construct a series of ideals
$$I\subset J_1\subset  \dots \subset J_{n-k-1}\subset J_{n-k}=L_1.$$
 of the Lie algebra $L_1=N+I$. Since the quotient  algebra $L_1/J_{n-k-1}$ is of rank 1 over $R$ and nilpotent, it is easy to see that $L_1/J_{n-k-1}$ is abelian (Lemma~\ref{abelian_factor}). Then $L_1'\subseteq J_{n-k-1}$, and  thus $\rank_R L_1'\leq n-1$. The latter  contradicts the equation $\rank_R L_1'=n$. The obtained contradiction shows that $\rank_R (L/I)'<\rank_R (L/I)$.
\end{proof}

\begin{lemma}\cite[Lemma 7]{Uzh}\label{operators}
Let $V$ be a nonzero vector space over $\mathbb K$ (not necessary finite dimensional). Let $T_1,\ T_2,\dots,T_k$ be pairwise commuting locally nilpotent operators on $V$. Then there exists a nonzero $v_0\in V$ such that $T_1(v_0)=T_2(v_0)=\dots=T_k(v_0)=0$.
\end{lemma}

\begin{lemma}\label{factor_center}
Let $L$ be a nonzero locally nilpotent subalgebra of  finite rank over $R$ of the Lie algebra $W(A)$. Let $I$ be a proper ideal of $L$ such that $I=RI\cap L$. Then the center of the quotient Lie algebra $L/I$ is nonzero.
\end{lemma}
\begin{proof}
Toward the contradiction, suppose the existence of nonzero subalgebras $L\subseteq W(A)$ with a proper ideal $I$ of $L$ that satisfy the conditions of the lemma and $Z(L/I)=0$. It is observed in Remark~\ref{note} that $\rank_R I<\rank_R L$. Let us choose among these Lie algebras a Lie algebra  $L$ with the least rank $\rank_R L/I.$  Then $\rank_R L/I>1$ (otherwise, in view of Lemma~\ref{abelian_factor},  the Lie algebra $L/I$ is abelian and has the nontrivial center).

Let $\rank_R L=n$, $\rank_R I=k$. Then $\rank_R (L/I)=n-k$. The derived subalgebra $(L/I)'=L'+I$ of the Lie algebra $L/I$ is of rank less then $\rank_R L/I$ by Lemma~\ref{factor_rank}. Set $M=R(L'+I)\cap L$. By Lemma~\ref{MakPetr} (2) $M$ is an ideal of $L$, and $\rank_R M=\rank_R (L'+I)< n$. It is easy to verify that $\rank_RM/I\leq \rank_R M<\rank_RL/I.$ The subalgebra $M$ of $L$ is locally nilpotent and thus $Z(M/I)\neq 0$ by our choice  of the Lie algebra $L$. Obviously, $Z(M/I)$ is a (possibly infinite-dimensional) vector space over $\mathbb K$. Note that the quotient Lie algebra $L/M$ is abelian and $\dim_F (FL/FM)=n-k$ by Lemma~\ref{loc_nilp} (2), where $F=F(L)$ is the field of constants for $L$.

Choose $D_1,\ D_2,\dots,D_{n-k}\in L$ such that the  cosets $D_1+FM, \ D_2+FM, \dots, \ D_{n-k}+FM$ form a basis of the vector space $FL/FM$ over $F$. Then linear operators $\ad D_1,\ \ad D_2,\dots,\ \ad D_{n-k}$ are locally nilpotent on the vector space $Z(M/I)$ over $\mathbb K$. Observe that $Z(M/I)$ is invariant of these linear operators as a characteristics ideal of the Lie algebra $M/I$.
Since $[\ad D_i, \ad D_j]=\ad [D_i,D_j]$ and $[D_i,D_j]\in M,$ linear operators $\ad D_i, \ad D_j$  pairwise commute on $Z(M/I)$ for $ i,j=1,2,\dots, n-k$. By Lemma~\ref{operators}, there exists a nonzero element $D_0+I\in M/I$ such that
$$\ad D_i(D_0 +I)=\ad D_i(D_0)+I=0+I \ \mbox{for all } i=1,2,\dots, n-k .$$
Hence $[FD_i, D_0+I]\subseteq FI$ for all $i=1,2,\dots, n-k$. Moreover,
$[M,D_0+I]\subseteq I$ and $[FM,D_0+I]\subseteq FI$. Since
$FL=FD_1+FD_2+\dots+FD_{n-k} +FM$, we obtain $[FL, D_0+I]\subseteq FI$. The latter states that $D_0+I\in Z(FL/FI)$. Therefore, in view of the condition $I=RI\cap L$, we get $D_0\in Z(L/I)$.
\end{proof}

\begin{corollary}\label{nonzero_center}(see  \cite{Uzh}, Theorem 1).
Let $L$ be a nonzero locally nilpotent subalgebra of  finite rank over $R$ of the Lie algebra $W(A)$. Then the center of the Lie algebra $L$ is nonzero.\end{corollary}

\begin{theorem}\label{th1}
Let $L$ be a locally nilpotent subalgebra of rank $n$ over $R$ of the Lie algebra $W(A)$. Let $F$ be the field of constants for $L$. Then
\begin{enumerate}
\item[(1)] $L$ contains a series of ideals
\begin{equation}\label{series}
0=L_0\subset L_1\subset \dots \subset L_n=L
\end{equation}
 such that $\rank_R L_s=s$ and the quotient Lie algebra $L_{s}/L_{s-1}$ is abelian for all $s=~1,2\dots,n$;

\item[(2)] There exists a basis $\{D_1, \ldots , D_n\}$ of $L$ over $R$ such that $$L_s=(RD_1+RD_2+\dots+RD_s)\cap L, \ [L,D_{s}]\subseteq L_{s-1} , \ s=1,2\dots,n ;$$

\item[(3)]  $\dim_F FL/FL_{n-1}=1$.
\end{enumerate}
\end{theorem}
\begin{proof}
{\em(1)-(2)} By Corollary~\ref{nonzero_center}, there exists a nonzero $D_1\in Z(L)$. Set $L_1=RD_1\cap L$. By Lemma~\ref{MakPetr} (2), $L_1$ is an ideal of $L$ of rank 1 over $R$. Assume that we have constructed basic elements $D_1, \ D_2, \dots, D_k$ of $L$ over $R$ such that $L_s=(RD_1+RD_2+\dots+RD_s)\cap L$ and $[L,D_s]\subseteq L_{s-1}$ for all $s=1,2,\dots,k$. Let us construct $D_{k+1}$. By Lemma~\ref{factor_center} the center $Z(L/L_{k})$ is nontrivial, so there exists $D_{k+1}\not \in L_k $ such that $D_{k+1}+L_k\in Z(L/L_k)$.
Then $[L,D_{k+1}]\subseteq L_k$, and one can easily see that $D_1, \ldots, D_k,\ D_{k+1}$ are linearly independent over $R$.
By Lemma~\ref{ideal} (2), $L_{k+1}=RD_{k+1}\cap L+L_{k}$ is an ideal of $L/L_k$. In view of the form of the ideal  $L_k$, we get that $L_{k+1}=(RD_1+\dots+RD_k+RD_{k+1})\cap L$ is an ideal of $L$ of rank $k+1$ over $R$.
We construct a series of ideals (\ref{series}) and a basis from the conditions of the theorem. Moreover, since $\rank_R L_{s+1}/L_{s}=1$ Lemma~\ref{abelian_factor}  implies that the quotient Lie algebras $L_{s+1}/L_{s}$ are abelian for all $s=0,1,\dots,n-1$.

 (3) The proof is analogous to the proof of Lemma~\ref{series_nilp_case}.
\end{proof}

\section{Locally nilpotent subalgebras of rank 3 of the Lie algebra $W(A)$}
In the following lemma the main results of \cite{PetravchukADM} are collected.
\begin{lemma}
\cite[Lemmas 8, 9]{PetravchukADM}\label{nilp_rk3}
Let $L$ be a nilpotent subalgebra of rank $3$ over $R$ of the Lie algebra $W(A)$. Let $Z(L)$ be the center of $L$ and $F$ the field of constants for $L$. If $\dim_F FL\geq 4,$ then there exist $a, \ b\in R$, integers $k\geq 1, \ n\geq 0, \ m \geq 1,$ and pairwise commuting elements $D_1,\ D_2, \ D_3 \in L$  such that the Lie algebra $FL$ is contained in the nilpotent subalgebra  $\widetilde L\subseteq W(A)$ of one of the following types:
 \begin{enumerate}
\item[(1)] If $\rank_R Z(L)=2$, then
$$\widetilde L=F\langle D_3, D_1,aD_1,\dots, \frac{a^k}{k!} D_1,D_2,aD_2,\dots, \frac{a^k}{k!} D_2\rangle,$$ where $D_1(a)=D_2(a)=0$ and $D_3(a)=1$.
\item[(2)]  If $\rank_R Z(L)=1$, then $\widetilde L$ is either the same as in (1), or
$$\widetilde L=F\langle D_3, D_2,aD_2,\dots, \frac{a^n}{n!} D_2, \Bigl\{\frac{a^ib^j}{i!j!}D_1\Bigr\}_{i,j=0}^{n,m} \rangle, $$ where $D_1(a)=D_2(a)=0$ and $D_3(a)=1$, $D_1(b)=D_3(b)=0$ and $D_2(b)=1$.
 \end{enumerate}
\end{lemma}

In \cite{Uzh}, the description of locally nilpotent subalgebras of $W(A)$ of ranks $1$ and $2$ was given.
\begin{lemma}\cite[Theorem 2]{Uzh}\label{Uzh1}
Let $L$ be a locally nilpotent subalgebra of the Lie algebra $W(A)$ and  $F$  the field of constants for $L.$  \begin{itemize}
\item[(1)] If $\rank_R L=1,$ then $L$ is abelian and $\dim_F FL=1.$
\item[(2)] If $\rank_R L=2,$ then $FL$ is either nilpotent finite dimensional over $F$, or infinite dimensional over $F$ and there exist
$D_1,\ D_2\in L,$  $a\in R$ such that
 $$FL=\langle D_2, D_1, aD_1, \dots, \frac{a^k}{k!}D_1,\dots \rangle,$$ where $[D_1,D_2]=0$, $D_1(a)=0,$ and $D_2(a)=1$.
\end{itemize}
\end{lemma}

\begin{theorem}
Let $L$ be a maximal (with respect to inclusion) locally nilpotent subalgebra of the Lie algebra $W(A)$ such that $\rank_R L=3$. Let $F$ be the field of constants for $L$. Then $FL=L$ and $L$ is a Lie algebra over $F$ of one of the following types:
\begin{enumerate}
\item[(1)] $L$ is a nilpotent Lie algebra of dimension  $3$  over $F$;
\item[(2)] $L=F\langle D_3,\{\frac{a^i}{i!}D_1\}_{i=0}^{\infty},\{\frac{a^i}{i!} D_2\}_{i=0}^{\infty}\rangle,$ where $D_1,\ D_2,\ D_3\in L$ and $a\in R$ such that $D_1(a)=D_2(a)=0$, $D_3(a)=1$, and $[D_i, D_j]=0$ for all $i,j=1,2,3$;
\item[(3)] $L=F\langle D_3,\{\frac{a^i}{i!} D_2\}_{i=0}^{\infty},\{\frac{a^ib^j}{i!j!}D_1\}_{i,j=0}^{\infty}\rangle,$ where $D_1, \ D_2, \ D_3\in L$ and $a,\ b\in R$ such that $D_1(a)=D_2(a)=0$, $D_3(a)=1$, and $D_1(b)=D_3(b)=0$, $D_2(b)=1$, and $[D_i, D_j]=0$ for all $i,j=1,2,3$.
\end{enumerate}
\end{theorem}
\begin{proof}
The Lie algebra $FL$ is locally nilpotent by Lemma~\ref{loc_nilp}. Therefore, the maximality of the subalgebra $L\subseteq W(A)$ implies $FL=L$. By Corollary~\ref{nonzero_center} $Z(L)\neq 0$. If $\rank_R Z(L)=3$, then one can easily see that $L$ is abelian. Thus, it follows from Lemma~\ref{MakPetr} that $L$ is the abelian Lie algebra of dimension $3$ over $F$ and $L$ is of type (1) from the conditions of the theorem.

{\em Case 1.}
 Let $\rank_R Z(L)=2$. Let us choose arbitrary elements $D_1,\ D_2\in Z(L)$ linearly independent over $R$ and set $I=(RD_1+RD_2)\cap L$. Then in view of Theorem~\ref{th1}, $I$ is an ideal of the Lie algebra $L$ of rank $2$ over $R$ and $\dim_F (FL/FI)=1$. It is easy to verify that $I$ is abelian. Indeed, take an arbitrary $D=r_1D_1+r_2D_2\in I$. Then
$$[D_1, D]=D_1(r_1)D_1+D_1(r_2)D_2=0, \ \ [D_2, D]=D_2(r_1)D_1+D_2(r_2)D_2=0,$$
whence it follows $r_1,\ r_2\in \Ker D_1\cap \Ker D_2$. Therefore, for all $D, D'\in I$ we get $[D, D']=0.$

Let us take an element $D_3\in L/I$.  It was proved above that  $FL=FI+FD_3$. Consider a nonabelian finitely generated (over $\mathbb K$) subalgebra $M$ of the Lie algebra $L$ such that $D_1,\ D_2, \ D_3 \in M$. Since $\rank_R M=3$ and $\rank_R Z(M)=2$, Lemma~\ref{nilp_rk3} implies that $FM$ is contained in some subalgebra $L_M$ of $W(A)$ of the form
$$L_M=F\langle D_3, D_1,aD_1,\dots, a^n/n! D_1,D_2,aD_2,\dots, a^n/n! D_2\rangle,$$
where $a\in R$ such that $D_3(a)=1$ and $D_1(a)=D_2(a)=0$. Then $M$ is contained in the locally nilpotent Lie algebra of the form
$$L_1=F\langle D_3, D_1,aD_1,\dots, a^n/n!D_1,\dots ,D_2,aD_2,\dots, a^n/n! D_2,\dots\rangle,$$
where $a\in R$ is defined by derivations $D_1,\ D_2, \ D_3$  up to a term from $F$. Since the subalgebra $M$ is an arbitrarily chosen in $L$, we have $L\subseteq L_1$. In view of the maximality of the Lie algebra $L$, we get that $L=L_1$ and $L$ is of type (2) from the conditions of the theorem.

{\em Case 2.} Let $\rank_R Z(L)=1$. If $\dim_F FL=3$, then the Lie algebra $L$ is nilpotent of dimension $3$ over $F$ and $L$ is of type (1) from the conditions of the theorem. Therefore, further we assume that $\dim_F FL\geq 4$. Take a nonzero $D_1\in Z(L)$. Then $I_1=RD_1\cap L$ is an ideal of the Lie algebra $L$ of rank 1 over $R$ (by Lemma~\ref{MakPetr}). The center of the quotient Lie algebra $L/I_1$ is nontrivial by Lemma~\ref{factor_center} and thus, we may choose a nonzero $D_2+I_1\in Z(L/I_1)$. By Theorem~\ref{th1}, $I_2=(RD_1+RD_2)\cap L$ is an ideal of the Lie algebra $L$, $\rank_R I_2=2$, and $\dim_F FL/FI_2=1$. Then for some $D_3\in FL\backslash FI_2$ we get $FL=FI_2+FD_3$. Moreover,  from the choice of $D_2$ it is easy to see that $[D_3,D_2]\in I_1$, so
$[D_3,D_2]=r_3D_1$ for some $r_3\in R$. In particular, this implies that derivations $D_2$ and  $D_3$ are commuting on $\Ker D_1$, i. e.
 \begin{equation}\label{eq_2}
 D_3(D_2(x))=D_2(D_3(x)) \mbox{ for all } x\in \Ker D_1.
 \end{equation}

Let us show that the ideal $FI_2$ is nonabelian.  Suppose this is not true and $FI_2$ is abelian. Since $\dim_F FL\geq 4$,  $\dim_F FI_2\geq 3$. Then there exists $D=r_1D_1+r_2D_2\in FI_2$ such that at least one of the coefficients $r_1, \ r_2$ is not in $F$.
From the obvious equalities
$$[D_1,D]=D_1(r_1)D_1+D_1(r_2)D_2=0,  \  \
[D_2,D]=D_2(r_1)D_1+D_2(r_2)D_2=0,$$
it follows $r_1, r_2\in \Ker D_1\cap \Ker D_2.$

Since at least one of $r_1, \ r_2$ not in $F$, either $D_3(r_1)\neq 0$ or $D_3(r_2)\neq 0$. Firstly, let $D_3(r_2)\neq 0$. The relation   $[D_3,D_2]=r_3D_1$ implies  that  for any integer $m\geq 1$ it holds
$$(\ad D_3)^m(D)=R_mD_1+D_3^m(r_2)D_2  \mbox{ for some } R_m\in R.$$
Since the  linear operator $\ad D_3$ is locally nilpotent on $L$, there exists an integer $k> 1$ such that $$D_3^{k-1}(r_2)\neq 0, \ D_3^{k}(r_2)=0.$$
Let us denote $r_0=D_3^{k-2}(r_2)$. Then $D_3(r_0)\neq 0$ and $D_3^2(r_0)=0$. Furthermore, it is easy to verify that $D_1(r_0)=D_2(r_0)=0$.
Set $a=\frac{r_0}{D_3(r_0)}\in R\backslash F$. One can easily checked that $D_1(a)=D_2(a)=0$ and $D_3(a)=1$.

Now let $D_3(r_2)=0$. Then $r_2\in F$, so $r_1\not \in F$ and $r_1D_1\in FI_2$ (because $r_2D_2\in FI_2$ and $D\in FI_2$). Note that $D_3(r_1)\neq 0$. Using the relation
$$(\ad D_3)^m(r_1D_1)=D_3^m(r_1)D_1, \ \ m\geq 1,$$ one can show (as in the case $r_2\not \in F$) that there exists $a\in R$ such that $D_1(a)=D_2(a)=0$ and $D_3(a)=1$.

Let us prove that $FI_1=F[a]D_1$, where $F[a]=\{f(a)| f(t)\in F[t]\}.$ Consider the  sum $L_1=L+F[a]D_1$. Since $[I_2, F[a]D_1]=0$ and $[D_3, F[a]D_1]\subseteq F[a]D_1$, $L_1$ is a subalgebra of $W(A)$ and  $L\subseteq L_1$. From the maximality of $L$, we get  $L_1=L$. Hence $F[a]D_1\subseteq FI_1$. Conversely, take an arbitrary element $rD_1\in I_1$. Then $D_1(r)=0$ and $D_2(r)=0$ since the ideal $FI_2$ is abelian by our assumption. The operator $\ad D_3$ acts locally nilpotently on $rD_1$, so $D_3^k(r)=0$ for some integer $k\geq 1$. One can show (using   \cite[Lemma 6]{PetravchukADM}) that $r$ is a linear combination over the field $F$ of elements  $1,\ a,\dots, a^t$ for some positive integer $t.$  Thus $rD_1\in F[a]D_1$. Therefore, $I_1\subseteq F[a]D_1$ and $FI_1=F[a]D_1$.

Since $[D_3,D_2]\in I_1$, we get that $[D_3, D_2]=f(a)D_1$ for some $f(t)\in F[t]$. The field $F$ is of characteristic zero, so there exists a polynomial $g(t)\in F[t]$ such that $g'(t)=f(t)$. Note that $D_3(g(a))=f(a)$ since $D_3(a)=1$. Set  $\widetilde D_2=D_2-g(a)D_1\in I_2$. Then $[D_3,\widetilde D_2]=0$, and  $\widetilde D_2$ has the same other properties as  the derivation $D_2$. Thus we may assume without loss of generality that $[D_3,D_2]=0$. Then $D_2\in Z(L)$ and $\rank_R Z(L)=2$, which contradicts our assumption. This means that the ideal $FI_2$ of the Lie algebra $FL$ is nonabelian.

Let us show that $L$ is a Lie algebra of type (3) from the conditions of the theorem. Consider  an arbitrary nonabelian finitely generated subalgebra $M$ of $FI_2$ such that $D_1, \ D_2 \in M$ (such a subalgebra does exist  because $FI_2$ is a nonabelian ideal of $L$). Denote by $N$ a subalgebra of $L$ generated by the Lie algebra $M$ and $D_3$. Then $\rank_R N=3$ and $N$ is a nonabelian Lie algebra that contains a nonabelian ideal $N_2$ of rank $2$ over $R$ such that $\dim_F FN/FN_2~=~1$. By Lemma~\ref{nilp_rk3}, there exist $a, \ b \in R$ such that $D_1(a)=D_2(a)=0$, $D_3(a)=1$, $D_1(b)=D_3(b)=0$, and $D_2(b)=1$ and $FN$ is contained in the Lie algebra $L_N$ of the form
$$L_N=F\langle D_3, D_2, aD_2, \dots, \frac{a^k}{k!} D_2, \Bigl\{ \frac{a^ib^j}{i!j!}D_1 \Bigl\}_{i,j=0}^{k,m}\rangle.$$
The elements $a,b \in R$ are uniquely determined by the derivations $D_1, \ D_2, \ D_3$ up to a summand in $F$. Thus $N$ is contained in the Lie  algebra $$L_2=F\langle D_3, D_2, aD_2, \dots, \frac{a^k}{k!} D_2, \dots, \Bigl\{ \frac{a^ib^j}{i!j!}D_1 \Bigl\}_{i,j=0}^{\infty}\rangle.$$
The Lie algebra $L_2$  is a locally nilpotent subalgebra of $W(A)$ of rank $3$ over $R$. Since $N$ is an arbitrarily chosen subalgebra of $L$, $L $ is contained in $L_2$. In view of  maximality of the Lie algebra $L$, we obtain $L=L_2$. The proof is complete.
\end{proof}

\begin{example}
Let $A=\mathbb K[x_1,x_2,x_3]$ and $R=\mathbb K(x_1,x_2,x_3)$. Then the Lie algebra $L=\mathbb K\langle x_1\frac{\partial}{\partial x_1}, x_2\frac{\partial}{\partial x_2}, x_3\frac{\partial}{\partial x_3}\rangle$ is abelian, $\rank_R L=3$, and $L$ is a maximal locally nilpotent subalgebra of the Lie algebra  $W_{3}(\mathbb K)$ (see \cite[Proposition 1]{SysakADM}).
\end{example}

\end{document}